\documentclass{amsart}
\usepackage[all,ps,cmtip]{xy}
\usepackage{amsmath, amssymb}
\usepackage{amscd}

\newtheorem{theorem}{Theorem}[section]
\newtheorem{lemma}[theorem]{Lemma}

\newtheorem{corollary}[theorem]{Corollary}

\newtheorem{proposition}[theorem]{Proposition}

\theoremstyle{definition}
\newtheorem{definition}[theorem]{Definition}

\theoremstyle{remark}
\newtheorem{remark}[theorem]{Remark}

\numberwithin{equation}{section}
\DeclareMathOperator{\Pic}{Pic}

\DeclareMathOperator{\coker}{coker}

\DeclareMathOperator{\rank}{rk}

\DeclareMathOperator{\im}{Im}

\DeclareMathOperator{\td}{td}

\DeclareMathOperator{\codim}{codim}

\DeclareMathOperator{\Coh}{Coh}

\DeclareMathOperator{\Supp}{Supp}

\DeclareMathOperator{\Spec}{Spec}

\DeclareMathOperator{\ch}{ch}

\DeclareMathOperator{\chr}{char}

\DeclareMathOperator{\reg}{reg}

\makeatletter
\@namedef{subjclassname@2020}{\textup{2020}
Mathematics Subject Classification}
\makeatother

\begin{document}
\title{Bogomolov type inequalities and Frobenius semipositivity}

\author{Hao Max Sun}

\address{Department of Mathematics, Shanghai Normal University, Shanghai 200234, People's Republic of China}

\email{hsun@shnu.edu.cn, hsunmath@gmail.com}
%\thanks{}


%    General info


\date{March 20, 2025}


\begin{abstract}
We prove Bogomolov type inequalities for high Chern characters of semistable sheaves satisfying certain Frobenius semipositivity.
The key ingredients in the proof are a high rank generalization of the asymptotic Riemann-Roch theorem and Langer's estimation theorem of the global sections of torsion free sheaves. These results give some Bogomolov type inequalities for semistable sheaves with vanishing low Chern characters. Our results are also applied to obtain inequalities of Chern characters of threefolds and varieties of small codimension in projective spaces and abelian varieties.
\end{abstract}

\maketitle

\setcounter{tocdepth}{1}
\tableofcontents
%\newpage

\section{Introduction}
Let $X$ be a smooth projective variety of dimension $d$ defined over an algebraically closed field $k$ of arbitrary characteristic and let $H$ be an ample divisor on $X$. The famous Bogomolov's inequality says that if $\chr(k)=0$, then $$\Delta(\mathcal{E})H^{d-2}:=(\ch^2_1(\mathcal{E})-2\ch_0(\mathcal{E})\ch_2(\mathcal{E}))H^{d-2}\geq0,$$ for any $\mu_H$-semistable sheaf $\mathcal{E}$ on $X$. It was proved by Bogomolov \cite{Bog} when $d=2$ (see also \cite{Gieseker}), and it can be easily generalized to higher dimensional case by the Mehta-Ramanathan restriction theorem. In the case of $\chr(k)>0$, Langer \cite{Langer1} proved that the same inequality holds for strongly $\mu_H$-semistable sheaves. Mehta and Ramanathan \cite{MR} showed that if $X$ satisfies $\mu_H^+(\Omega_X^1)\leq0$, then all $\mu_H$-semistable sheaves on $X$ are strongly $\mu_H$-semistable. Thus Bogomolov's inequality holds on such an $X$. See also \cite{Sun3} and \cite{Liu} for Bogomolov's inequality on product type varieties.

Bogomolov's inequality has many interesting and important applications, such as
the positivity of adjoint linear systems \cite{Reider, BS}, extension of linear systems on a divisor \cite{Pao}, divisors on the moduli space of curves \cite{Mo1}, slope inequalities \cite{Mo2}, Cayley-Bacharach property \cite{Tan}, vanishing theorems for torsion free sheaves \cite{Sun}, etc. These applications motivate us to generalize Bogomolov's inequality to higher Chern classes. In \cite{BMT}, Bayer, Macr\`i and Toda give a conjectural construction of Bridgeland stability conditions on the derived category of threefolds. The construction depends on a conjectural Bogomolov type inequality for the third Chern character of certain stable complexes (see also \cite{Sun1, Sun2} for some similar conjectural inequalities on fibred threefolds).  The conjectural inequality has been confirmed on some threefolds. We refer to \cite[$\cdots$]{BMS, BMSZ, Li1} for the precise results. But it seems very difficult to prove for a general threefold. We also refer to \cite{DRY} for a conjectural approach to sufficient conditions for Chern classes to correspond to stable bundles on Calabi-Yau threefolds.

In this paper, we will try to generalize Bogomolov's inequality in a different way. Let us consider a splitting vector bundle
$$\mathcal{G}=\mathcal{O}_X(a_1H)\oplus\mathcal{O}_X(a_2H)\oplus\cdots\oplus\mathcal{O}_X(a_rH),$$ where $a_i\in \mathbb{Z}$. In general, it is not semistable. A simple computation shows that $$H^{d-2}\Delta(\mathcal{G})=H^d((a_1+\cdots+a_r)^2-r(a_1^2+\cdots+a_r^2)).$$
Thus the Cauchy–Schwarz inequality gives the opposite of the Bogomolov inequality: $H^{d-2}\Delta(\mathcal{G})\leq0$. This observation inspires us that in order to generalize Bogomolov's inequality to higher Chern classes one need to consider a higher degree generalization of the Cauchy–Schwarz inequality. It turns out that when $t\geq1$ and $a_i\geq0$ for $1\leq i\leq r$, the power mean inequality $$(\frac{a_1+\cdots+a_r}{r})^t\leq\frac{a_1^t+\cdots+a_r^t}{r}$$ naturally extends the Cauchy–Schwarz inequality. For the splitting bundle $\mathcal{G}$, the power mean inequality is equivalent to
\begin{equation}\label{pm}
\frac{(H^{d-1}\ch_1(\mathcal{G}))^t}{t!(H^dr)^{t-1}}\leq H^{d-t}\ch_t(\mathcal{G}).
\end{equation}
Since the power mean inequality in general is not true without the positivity assumption $a_i\geq0$, one may expect that the opposite of the inequality (\ref{pm}) holds for a semistable sheaf with certain positivity. Our main result meets our expectation:

\begin{theorem}\label{Bog1}
Assume that $\chr(k)=0$. Let $\mathcal{E}$ be a rank $r$ torsion free sheaf on $X$ with $\phi_{sp}(\mathcal{E})\leq1$. If $\mathcal{E}$ is
$\mu_H$-semistable, then
$$H^{d-t}\ch_t(\mathcal{E})\leq\max\left\{\frac{(H^{d-1}\ch_1(\mathcal{E}))^t}{t!(H^dr)^{t-1}}, 0\right\}$$ for any $1\leq t\leq d$.
\end{theorem}

The positivity condition $\phi_{sp}(\mathcal{E})\leq1$ here is a certain Frobenius semipositivity, which is a variant of the Frobenius amplitude introduced by Arapura \cite{Ara1}. It measures the positivity of $\mathcal{E}$ under the Frobenius pull backs $(F_X^n)^*$ for $n\gg0$. To be precise, if $\chr(k)>0$ the sheaf $\mathcal{E}$ satisfies $\phi_{sp}(\mathcal{E})\leq 1$ if and only if there exist an integer $m$ and a very ample divisor $L$ so that $H^{j}((F_X^n)^*\mathcal{E}(iL))=0$ for any $n\geq0$, $i\geq m$ and $j\geq2$ (see Remark \ref{rm4.2}).  If $\chr(k)=0$, we say $\phi_{sp}(\mathcal{E})\leq1$ if $\phi_{sp}(\mathcal{E}_q)\leq1$ for a general $q$ over an arithmetic thickening of $(X, \mathcal{E})$ (see Section \ref{S3}).

The strategy of the proof of the theorem above is the following. Firstly, we prove two Fujita type vanishing theorems for F-semipositivity (Theorem \ref{Fujita3} and \ref{Fujita4}), which are generalizations of the results of Keeler \cite{Ke}. Then we show that these vanishing theorems give an asymptotic Riemann-Roch theorem for sheaves under the Frobenius pull backs (Corollary \ref{cor4.13}). From this and Langer's estimation theorem of the global sections of torsion free sheaves, one obtains
Theorem \ref{Bog1} by the standard spreading out technique.

In positive characteristic case, some inequalities similar to that in Theorem \ref{Bog1} are showed in Theorem \ref{Bog3}. As a corollary of Theorem \ref{Bog1} and \ref{Bog3}, one can obtain Bogomolov type inequalities without positivity assumptions.

\begin{corollary}\label{Bog4}
Let $\mathcal{E}$ be a rank $r$ strongly $\mu_H$-semistable torsion free sheaf on $X$. Let $t$ be an integer satisfying $2\leq t\leq d$. If $$(H^{d-1}\ch_1(\mathcal{E}))^i-i!(H^dr)^{i-1}(H^{d-i}\ch_i(\mathcal{E}))=0$$ for all $1\leq i\leq t-1$, then we have $$(H^{d-1}\ch_1(\mathcal{E}))^t-t!(H^dr)^{t-1}(H^{d-t}\ch_t(\mathcal{E}))\geq0.$$ In particular, if $H^{d-i}\ch_i(\mathcal{E})=0$ for all $1\leq i\leq t-1$, then $H^{d-t}\ch_t(\mathcal{E})\leq0$.
\end{corollary}

Since $(H^{d-1}\ch_1(\mathcal{E}))^i-i!(H^dr)^{i-1}(H^{d-i}\ch_i(\mathcal{E}))=0$ always holds true for $i=1$, the corollary gives
$$(H^{d-1}\ch_1(\mathcal{E}))^2-2H^dr(H^{d-2}\ch_2(\mathcal{E}))\geq0,$$ which is a weak form of the Bogomolov inequality. When $$H^{d-1}\ch_1(\mathcal{E})=H^{d-2}\ch_2(\mathcal{E})=0,$$ under the assumption that $\mathcal{E}$ is reflexive, Langer \cite{Langer3} (see also \cite{FL}) proved the numerical triviality of $\ch_i(\mathcal{E})$ for all $i>0$. Our inequalities in Corollary \ref{Bog4} is weaker than Langer's under the reflexive assumption, but it works for the more general torsion free case. One notices that Langer's result is not valid without the assumption of reflexivity. For example, one takes $\mathcal{E}$ to be the ideal sheaf of a codimension three subvariety of $X$. Then one sees that $H^{d-1}\ch_1(\mathcal{E})=H^{d-2}\ch_2(\mathcal{E})=0$ but $H^{d-3}\ch_3(\mathcal{E})<0$.

We also prove an analog of Theorem \ref{Bog1} for abelian varieties:
\begin{theorem}\label{Bog2}
Let $(A, \Theta)$ be a polarized complex abelian variety of dimension $d$, and let $\mathcal{E}$ be a rank $r$ torsion free sheaf on $A$. If $\mathcal{E}$ is a $GV_{-1}$ sheaf, then
$$\ch_d(\mathcal{E})\leq\max\left\{\frac{r(\mu_{\Theta}^+(\mathcal{E}))^d}{d!({\Theta}^d)^{d-1}}, 0\right\}.$$
\end{theorem}
Here the positivity condition that $\mathcal{E}$ is a $GV_{-1}$ sheaf is introduced by Pareschi and Popa \cite{PP}. It equivalents to say that $\mathcal{E}$ satisfies Generic Vanishing with index $-1$, i.e.,
$$\codim \{\alpha\in\Pic^0(A): h^i(\mathcal{E}\otimes\alpha)>0 \}\geq i-1$$ for all $i$.

\subsection*{Applications}
These Bogomolov type inequalities above can be used to give a bunch of inequalities of Chern characters of threefolds and varieties of small codimension in projective spaces and abelian varieties (Theorem \ref{Chern1} and \ref{Chern2}). In particular, for a threefold we obtain the following:
\begin{proposition}[Corollary \ref{cor6.2}]
Let $X$ be a complex projective smooth threefold. Let $H$ be a very ample divisor on $X$. If $K_X$ is ample, then $$\ch_3\Big(\Omega_X^1\otimes\omega_X(5H)\Big)\leq\frac{\Big(4K_X^3+15K_X^{2}H\Big)^3}{54(K_X^3)^2}.$$
If $K_X=0$, then one has
$$c_3(X)+10Hc_2(X)\geq0.$$
\end{proposition}
In \cite{KW} Kanazawa and Wilson obtained many interesting bounds for $c_2$ and $c_3$ of a Calabi-Yau threefold. See also \cite{CK, Hun} for some explicit bounds on $c_3$ for certain types of Calabi–Yau threefolds. One can compare our inequality with the bounds of Kanazawa and Wilson:
\begin{theorem}(\cite[Theorem 1.2]{KW})
Let $(X,H)$ be a very amply polarized Calabi–Yau threefold, i.e., $H$ is a very ample divisor on $X$. Then the following inequality
holds:
$$-36H^3-80\leq \frac{c_3(X)}{2}\leq 6H^3+40.$$
\end{theorem}

For Fano threefolds, Theorem \ref{Chern1} gives:
\begin{proposition}[Corollary \ref{cor-Fano}]
Let $X$ be a complex smooth Fano threefold and $l_X$ the Fano index of $X$. Assume that $-K_X\sim l_XL$ for an ample divisor $L$ in $X$. Then one has
$$(\frac{13}{2}l_X-25)Lc_2(X)-\frac{1}{2}c_3(X)\leq(\frac{127}{6}l_X^3-\frac{325}{2}l_X^2+\frac{625}{2}l_X)L^3.$$ If $\Omega_X^1$ is $\mu_L$-semistable and $l_X\leq3$, we have
$$(\frac{13}{2}l_X-25)Lc_2(X)-\frac{1}{2}c_3(X)\leq(\frac{58}{27}l_X^3-\frac{25}{3}l_X^2)L^3.$$
\end{proposition}
We refer to \cite{Kaw}, \cite{IJL} and \cite{LL} for the Kawamata–Miyaoka type inequality of $c_1$ and $c_2$ of $\mathbb{Q}$-Fano varieties. We also obtain a general inequality for torsion free sheaves on threefolds.
\begin{theorem}[Theorem \ref{Thm6.5}]
Let $(X, H)$ be a polarized complex projective smooth threefold. Let $\mathcal{E}$ be a torsion free sheaf of rank $r$ on $X$. If $\det\mathcal{E}$ is ample or trivial and $\mathcal{E}$ is generically globally generated in dimension one, then one has $$\ch_3(\mathcal{E}\otimes\det\mathcal{E})\leq\frac{r(\mu_H^+(\mathcal{E}\otimes\det\mathcal{E}))^3}{6(H^3)^2}.$$
\end{theorem}

As a corollary, we also obtain a Miyaoka-Yau type inequality:
\begin{corollary}[Corollary \ref{cor6.4}]
Let $X$ be a complex smooth threefold. Assume that $\Omega^1_X$ is generically globally generated in dimension one. If $K_X$ is ample, then one has
$$c_3(X)\geq3c_1(X)c_2(X)-\frac{26}{27}c_1^3(X).$$
\end{corollary}
If the cotangent bundle $\Omega_X^1$ of the threefold $X$ is nef, we can apply to it the Fulton-Lazarsfeld theory for nef vector bundles developed by Demailly, Peternell and Schneider \cite{DPS}. By \cite[Theorem 2.5]{DPS}, one obtains
\begin{eqnarray*}
  c_3(X) &\leq& 0 \\
  c_1(X)c_2(X) &\leq& c_3(X)\\
  c_1^3(X)+c_3(X)&\leq&2c_1(X)c_2(X).
\end{eqnarray*}
One can compare these inequalities with ours. We refer to \cite{DS} for some general Miyaoka-Yau type inequalities.

\subsection*{Organization of the paper}
Our paper is organized as follows. In Section \ref{S2}, we review
basic notions and properties of the classical stability for coherent
sheaves. Then in Section \ref{S3} we recall
the definition and properties of Frobenius amplitude and partial Castelnuovo-Mumford regularity. In Section \ref{S4} we slightly generalize the concept of the Frobenius semipositive sheaves introduced by Arapura \cite{Ara1}, and prove two Fujita type vanishing theorems for F-semipositivity (Theorem \ref{Fujita3} and \ref{Fujita4}), which are generalizations of the results of Keeler \cite{Ke}. Then we show that these vanishing theorems give an asymptotic Riemann-Roch theorem for sheaves under the Frobenius pull backs (Corollary \ref{cor4.13}).
Theorem \ref{Bog1}, \ref{Bog2} and Corollary \ref{Bog4} will be proved in Section \ref{S5}. The applications of our main
theorems will be discussed in Section \ref{S6}.

\subsection*{Notation}
Given a field $k$, we write $\chr(k)$ for its characteristic. Let $X$ be a smooth projective variety defined over
$k$. We denote by $T_X$ and $\Omega_X^1$ the tangent bundle and cotangent bundle of $X$, respectively. $K_X$ and $\omega_X$ denote the canonical divisor and canonical sheaf of $X$, respectively. We denote by $D^b(X)$ the bounded derived category of coherent sheaves on $X$.
We write $\ch(\mathcal{E})$ and $c(\mathcal{E})$ for the Chern character and Chern class
of a coherent sheaf $\mathcal{E}$ on $X$, respectively. We denote by $c_i(X):=c_i(T_X)$ for the $i$-th Chern class of $X$. We also write $H^j(\mathcal{F})$ ($j\in \mathbb{Z}_{\geq0}$) for the cohomology groups of a sheaf
$\mathcal{F}\in\Coh(X)$ and $h^j(\mathcal{F})$ for the dimension of $H^j(\mathcal{F})$. For a
sheaf $\mathcal{G}\in\Coh(X)$, we denote by $\mathcal{G}^*:=\mathcal{H}om(\mathcal{G}, \mathcal{O}_X)$ the
dual sheaf of $\mathcal{G}$ and by $\Delta(\mathcal{G}):=\ch^2_1(\mathcal{G})-2\ch_0(\mathcal{G})\ch_2(\mathcal{G})$
the discriminant of $\mathcal{G}$.

\subsection*{Acknowledgments}
%The author is grateful to the referee for his or her valuable
%comments and pointing out some errors in the earlier version of this paper. 
The author would like to thank Wenfei Liu for his
interest and discussions. The author was supported by National Natural Science Foundation of China (Grant No.
12371047, 11771294) and Natural Science Foundation of Shanghai
(Grant No. 23ZR1447000).

\section{Slope stability for sheaves}\label{S2}
Throughout this section, we let $X$ be a smooth projective variety of dimension $d$ defined over an algebraically closed field $k$
with $\chr(k)=p$ and $H$ be a fixed ample divisor on $X$. Denote by $F_X$ the absolute Frobenius morphism of $X$, if $p>0$. We will review some basic properties of slope stability for coherent sheaves in this section.

\begin{definition}
The support of a coherent sheaf $\mathcal{F}$ on $X$ is the closed set $$\Supp(\mathcal{F})=\{x\in X: \mathcal{F}_x\neq0\}.$$ Its
dimension is called the dimension of the sheaf $\mathcal{F}$ and is denoted by $\dim(\mathcal{F})$. We denote by $\det\mathcal{F}$ the determinant invertible sheaf of $\mathcal{F}$.
\end{definition}
The annihilator ideal sheaf of $\mathcal{F}$, i.e., the kernel of $\mathcal{O}_X\rightarrow\mathcal{H}om(\mathcal{F}, \mathcal{F})$, defines a subscheme structure on $\Supp(\mathcal{F})$. If $\dim\mathcal{F}\leq d-2$, then $\det\mathcal{F}\cong\mathcal{O}_X$.

\begin{definition}
Let $\mathcal{F}$ be a coherent sheaf on $X$. We say $\mathcal{F}$ is generically globally generated in dimension $n$ if the cokernel of the evaluation morphism $$H^0(\mathcal{F})\otimes\mathcal{O}_X\xrightarrow{ev_{\mathcal{F}}}\mathcal{F}$$ is a torsion sheaf with dimension $\leq n$.
\end{definition}

We define the slope $\mu_H$ of a
coherent sheaf $\mathcal{E}\in \Coh(X)$ by
\begin{eqnarray*}
\mu_H(\mathcal{E})= \left\{
\begin{array}{lcl}
+\infty,  & &\mbox{if}~\rank(\mathcal{E})=0,\\
&&\\
\frac{H^{d-1}\ch_1(\mathcal{E})}{\rank(\mathcal{E})}, & &\mbox{otherwise}.
\end{array}\right.
\end{eqnarray*}

\begin{definition}\label{def2.1}
A coherent sheaf $\mathcal{E}$ on $X$ is $\mu_{H}$-(semi)stable if, for all non-zero subsheaves
$\mathcal{F}\hookrightarrow \mathcal{E}$, we have
$$\mu_{H}(\mathcal{F})<(\leq)\mu_{H}(\mathcal{E}/\mathcal{F}).$$ We say a $\mu_{H}$-semistable sheaf $\mathcal{E}$ is strongly $\mu_{H}$-semistable if either $p=0$ or $p>0$ and all the
Frobenius pull backs of $\mathcal{E}$ are $\mu_H$-semistable.
\end{definition}
The Harder-Narasimhan filtrations (HN-filtrations, for short)
with respect to $\mu_H$-stability exist in $\Coh(X)$: given a
non-zero sheaf $\mathcal{E}\in\Coh(X)$, there is a filtration
$$0=\mathcal{E}_0\subset \mathcal{E}_1\subset\cdots\subset \mathcal{E}_m=\mathcal{E}$$
such that: $\mathcal{G}_i:=\mathcal{E}_i/\mathcal{E}_{i-1}$ is $\mu_H$-semistable, and
$\mu_H(\mathcal{G}_1)>\cdots>\mu_H(\mathcal{G}_m)$. We set $\mu^+_H(\mathcal{E}):=\mu_H(\mathcal{G}_1)$ and $\mu^-_H(\mathcal{E}):=\mu_H(\mathcal{G}_m)$.

Mehta and Ramanathan proved the following proposition.
\begin{proposition}
If $X$ satisfies $\mu_H^+(\Omega_X^1)\leq0$, then all
$\mu_H$-semistable sheaves on $X$ are strongly $\mu_H$-semistable.
\end{proposition}
\begin{proof}
See \cite[Theorem 2.1]{MR}.
\end{proof}

We will use the following estimation theorem of Langer.
\begin{theorem}\label{Langer1}
Assume that $H$ is very ample. Let $\mathcal{E}$ be a rank $r$ torsion free sheaf on $X$. Then
\begin{eqnarray*}
h^0(X, \mathcal{E})\leq\left\{
\begin{array}{lcl}
rH^d\binom{\frac{\mu_H^+(\mathcal{E})}{H^d}+d+f(r)}{d},  & &\mbox{if}~\mu_H^+(\mathcal{E})\geq0,\\
&&\\
0, & &\mbox{otherwise},
\end{array}\right.
\end{eqnarray*}
where $f(r)=-1+\sum_{i=1}^{r}\frac{1}{i}$.
\end{theorem}
\begin{proof}
See \cite[Theorem 3.3]{Langer2}.
\end{proof}

Now let us recall another theorem of Langer on the semistability of Frobenius pull backs.

\begin{theorem}\label{Langer2}
Let $B$ be a nef divisor on $X$ such that $T_X(B)$ is globally generated. Let $\mathcal{E}$ be a torsion free sheaf on $X$ of rank $r$. Then for any $n\geq0$, we have $$\frac{\mu_H^+((F_X^n)^*\mathcal{E})}{p^n}-\mu_H^+(\mathcal{E})\leq\frac{r-1}{p-1}BH^{d-1}.$$
\end{theorem}
\begin{proof}
See \cite[Corollary 2.5]{Langer1}.
\end{proof}

\section{Frobenius amplitude and partial regularity}\label{S3}
This section reviews some definitions and facts of Frobenius amplitude introduced by Arapura in \cite{Ara1} and partial Castelnuovo–Mumford regularity.

Throughout this section, we let $k$ be a field with $\chr(k)=p$ and $X$ be a projective scheme over $k$. A diagram over $\Spec k$ is defined to be a collection of $k$-schemes $X_i$, morphism $X_i\rightarrow X_j$, coherent sheaves $\mathcal{E}_{i,l}$ on $X_i$ and morphisms between the pullbacks and pushforwards of these sheaves.

Now assume that $p=0$. Given a diagram $D$ over $\Spec k$, an arithmetic thickening (or simply just thickening) of it is a finitely generated $\mathbb{Z}$-subalgebra $A\subset k$ and a diagram $\widetilde{D}$ over $\Spec A$, so that $D$ is isomorphic to the fiber product $\widetilde{D}\times_{\Spec A}\Spec k$. We say a thickening $\widetilde{D}_0\rightarrow\Spec A_0$ refines the thickening $\widetilde{D}\rightarrow\Spec A$ if there is a homomorphism $A\rightarrow A_0$ and an isomorphism between $\widetilde{D}_0$ and $\widetilde{D}\times_{\Spec A}\Spec A_0$.

We have the following standard facts (see \cite[Lemma 1.1]{Ara1}).

\begin{lemma}\label{lemma3.1}
Any diagram over $\Spec k$ has a thickening. Any two thickenings over $\Spec k$ have a common refinement.
\end{lemma}

When $p>0$, we denote by $F_X$ the absolute Frobenius morphism of $X$. Given a coherent sheaf $\mathcal{E}$ on $X$, denote $(F_X^n)^*\mathcal{E}$ by $\mathcal{E}^{(p^n)}$. For a thickening $(\widetilde{X}, \widetilde{\mathcal{E}})\rightarrow\Spec A$ of $(X, \mathcal{E})$, we write $p_q=\chr(A/q)$, $X_q$ for the fiber and $\mathcal{E}_q=\widetilde{\mathcal{E}}|_{X_q}$ for each closed point $q\in \Spec A$.

\subsection{Frobenius amplitude}\label{S3.1}
The notion of Frobenius amplitude was given by Arapura which provides a cohomological measure of the positivity of a sheaf.
\begin{definition}
Let $\mathcal{F}$ be a coherent sheaf on $X$. If $p>0$, the Frobenius amplitude (or simply F-amplitude) $\phi_a(\mathcal{F})$ of $\mathcal{F}$ is the smallest nonnegative integer such that for any locally free sheaf $\mathcal{E}$, there exists an integer $n_0$ such that $H^i(X, \mathcal{E}\otimes\mathcal{F}^{(p^n)})=0$ for all $i>\phi_a(\mathcal{F})$ and $n>n_0$. If $p=0$, the F-amplitude $\phi_a(\mathcal{F})$ is defined to be the minimum of $\max_q\phi_a(\mathcal{F}_q)$ over all thickenings. We say $\mathcal{F}$ is F-ample if $\phi_a(\mathcal{F})=0$.
\end{definition}

The following lemma shows that F-amplitude can be only tested by tensoring with ample line bundles.
\begin{lemma}\label{lemma3.3}
Assume that $p>0$. Fix an ample line bundle $\mathcal{L}$. Then $\phi_a(\mathcal{F})\leq N$ if and only if for any integer $m$ there exists an integer $n_0$ such that $$H^i(X, \mathcal{L}^m\otimes\mathcal{F}^{(p^n)})=0$$ for all $i>N$ and $n>n_0$.
\end{lemma}
\begin{proof}
See \cite[Corollary 2.3]{Ara1}.
\end{proof}

The characterization of F-ampleness of a line bundle is due to \cite[Lemma 2.4]{Ara1}:
\begin{lemma}\label{lemma3.4}
A line bundle is F-ample if and only if it is ample.
\end{lemma}

For higher rank ample vector bundles, \cite[Theorem 5 and Proposition D.3]{Ara1} gave estimates on their F-amplitude:
\begin{theorem}\label{ample}
Assume that $X$ is integral with $\dim X>0$. Let $\mathcal{E}$ be an ample vector bundle on $X$ of rank $r$. Then $\phi_a(\mathcal{E})<\dim X$. Furthermore, if $p=0$ we have $\phi_a(\mathcal{E})<r$.
\end{theorem}

The concept of Frobenius amplitude behave well under base change.
\begin{lemma}\label{ext1}
Let K be an extension of $k$ and $\mathcal{E}$ a coherent sheaf on $X$. Then we have $\phi_a(\mathcal{E})=\phi_a(\mathcal{E}\otimes_kK)$.
\end{lemma}
\begin{proof}
The equality is proved in \cite[Lemma C.5.]{Ara1} and \cite[Lemma 3.8]{Ke}.
\end{proof}

The following theorem proved by Keeler is a generalization of Fujita's vanishing theorem.
\begin{theorem}\label{Fujita1}
Let $\mathcal{L}$ be an ample line bundle on $X$. Given any coherent sheaf $\mathcal{F}$ on $X$, there exists an integer $m(\mathcal{F}, \mathcal{L})$ such that
$$H^i(X, \mathcal{F}\otimes\mathcal{L}^m\otimes\mathcal{G})=0$$ for all $i>t$, $m\geq m(\mathcal{F}, \mathcal{L})$ and any locally free sheaf $\mathcal{G}$ on $X$ with $\phi_a(\mathcal{G})\leq t$.
\end{theorem}
\begin{proof}
See \cite[Theorem 1.4]{Ke}.
\end{proof}

We may have another generalization of Fujita's vanishing theorem.
\begin{theorem}\label{Fujita2}
Assume that $X$ is smooth. Let $\mathcal{L}$ be an ample line bundle on $X$. Given any locally free sheaf $\mathcal{F}$ on $X$, there exists an integer $m(\mathcal{F}, \mathcal{L})$ such that
$$H^i(X, \mathcal{F}\otimes\mathcal{L}^m\otimes\mathcal{G})=0$$ for all $i>t$, $m\geq m(\mathcal{F}, \mathcal{L})$ and any coherent sheaf $\mathcal{G}$ on $X$ with $\phi_a(\mathcal{G})\leq t$.
\end{theorem}
\begin{proof}
The vanishing has been proved in \cite[Theorem 4.5]{Ke} under the assumption that $k$ is perfect. Since the F-amplitude and the dimension of the cohomology are invariant under field extensions, the desired vanishing follows over any field.
\end{proof}

%Arapura also gave the basic vanishing theorem:
%\begin{theorem}
%Suppose that $p=0$ and $X$ is smooth. For a coherent sheaf $\mathcal{E}$ on $X$, we have $H^i(X, \omega_X\otimes\mathcal{E})=0$ for %$i>\phi_a(\mathcal{E})$.
%\end{theorem}
%\begin{proof}
%The conclusion is a special case of \cite[Corollary 8.5]{Ara1}.
%\end{proof}

\subsection{Partial Castelnuovo–Mumford regularity}\label{S3.2}
Castelnuovo–Mumford regularity gives a quantitative measure of the algebraic complexity of a coherent sheaf. We will recall the definition and basic properties (e. g. \cite[sec. 2]{Ke}).

\begin{definition}
Let $\mathcal{O}_X(1)$ be a very ample line bundle on $X$, and $t$ a nonnegative integer. We say a coherent sheaf $\mathcal{E}$ on $X$ is $(m, t)$-regular if $$H^i(X, \mathcal{E}(m-i))=0$$ for $i>t$. The $t$-regularity $\reg_t(\mathcal{E})$ of $\mathcal{E}$ is the minimum $m$ such that $\mathcal{E}$ is $(m, t)$-regular. If $\mathcal{E}$ is $(m, 0)$-regular, we simply say $\mathcal{E}$ is $m$-regular.
\end{definition}

The $0$-regularity is the usual Castelnuovo–Mumford regularity (see \cite[sec. 1.8]{Laz}). when $t>0$, we say $t$-regularity partial Castelnuovo–Mumford regularity. Here we use $t$-regularity instead of just $0$-regularity because it is necessary for the study of F-amplitude. This concept of $t$-regularity shares a basic property with $0$-regularity.

\begin{lemma}\label{pCM}
Let $\mathcal{E}$ be a coherent sheaf on $X$. If $\mathcal{E}$ is $(m, t)$-regular, then $\mathcal{E}$ is $(n, t)$-regular for all
$n\geq m$.
\end{lemma}
\begin{proof}
The proof is essentially the same as that of $0$-regularity (cf. \cite[Lemma 2.2]{Ke}).
\end{proof}

The related notion of the level $\lambda(\mathcal{E})$ of a sheaf was introduced in \cite[sec. 1]{Ara2}. It measures the deviation from $0$-regular sheaves. The following lemma is helpful for us to compute the cohomology groups of a sheaf.

\begin{lemma}\label{lemma3.8}
Let $\mathcal{E}$ be a $m$-regular sheaf on $X$ with respect to a very ample line bundle $\mathcal{O}_X(1)$. Then for any $N\geq0$, there exist vector spaces $V_i$ and a resolution
$$V_N\otimes\mathcal{O}_X(-m-NR)\rightarrow\cdots \rightarrow V_1\otimes\mathcal{O}_X(-m-R)\rightarrow V_0\otimes\mathcal{O}_X(-m)\rightarrow\mathcal{E}\rightarrow0$$ where $R=\max\{1, \reg_0(\mathcal{O}_X)\}$.
\end{lemma}
\begin{proof}
See \cite[Corollary 3.2]{Ara1}.
\end{proof}

\section{Frobenius semipositivity}\label{S4}
The concept of Frobenius semipositive sheaves (or simply F-semipositive sheaves) was also introduced in \cite{Ara1}, which is obtained by relaxing the condition for F-ampleness. We give a slight generalization of this concept in this section in order to adapt our situation. Let $k$ be a field with $\chr(k)=p$ and $X$ a projective scheme over $k$ in this section. We keep the same notations as that in Section \ref{S3}.

\subsection{Basic definitions and properties}

\begin{definition}
Let $\mathcal{O}_X(1)$ be a very ample line bundle on $X$ and $\mathcal{E}$ a coherent sheaf on $X$. If $p>0$, the F-semipositivity $\phi_{sp}(\mathcal{E})$ of $\mathcal{E}$ is the smallest nonnegative integer such that the $\phi_{sp}(\mathcal{E})$-regularities of $\mathcal{E}^{(p^n)}$ are bounded. If $p=0$, the F-semipositivity $\phi_{sp}(\mathcal{E})$ is defined to be the minimum of $\max_q\phi_{sp}(\mathcal{E}_q)$ over all thickenings. We say $\mathcal{E}$ is F-semipositive if $\phi_{sp}(\mathcal{E})=0$.
\end{definition}

\begin{remark}\label{rm4.2}
If $p>0$, it follows from Lemma \ref{pCM} that $\phi_{sp}(\mathcal{E})\leq N$ if and only if there exists an integer $m$ so that $H^{j}(\mathcal{E}^{(p^n)}(i))=0$ for any $n\geq0$, $i\geq m$ and $j>N$.
If $p=0$, by Lemma \ref{lemma3.1}, one sees that $\phi_{sp}(\mathcal{E})\leq N$ (resp. $\phi_{a}(\mathcal{E})\leq N$) if and only if $\phi_{sp}(\mathcal{E}_q)\leq N$ (resp. $\phi_{a}(\mathcal{E}_q)\leq N$) for almost all $q$ for a given thickening.
\end{remark}

From the definition of F-amplitude and F-semipositivity, one immediately obtains that:
\begin{lemma}\label{lemma3.10}
For any coherent sheaf $\mathcal{E}$ on $X$, we have $\phi_{sp}(\mathcal{E})\leq \phi_{a}(\mathcal{E})$. In particular, an F-ample sheaf is F-semipositive.
\end{lemma}

The proposition below is a generalization of \cite[Corollary 3.9]{Ara1}.
\begin{proposition}
Let $\mathcal{E}$ be a coherent sheaf on $X$. Then
$\phi_{sp}(\mathcal{E})$ is independent of the choice of very ample line bundles.
\end{proposition}
\begin{proof}
Let $\mathcal{L}_1$ and $\mathcal{L}_2$ be two very ample line bundles on $X$. We write $\phi^{i}_{sp}(\mathcal{E})$ for the F-semipositivity of $\mathcal{E}$ with respect to $\mathcal{L}_i$, where $i=1$ or $2$. Using a thickening, we may assume $p>0$. Let $R_0$ be the $0$-regularity of  $\mathcal{O}_X$ with respect to $\mathcal{L}_2$ and $d=\dim X$.

By the definition of $\phi^{2}_{sp}(\mathcal{E})$, there exists an integer $m_2$ such that $$H^i(X, \mathcal{E}^{(p^n)}\otimes\mathcal{L}^{m-i}_2)=0$$ for any $n\geq0$, $m\geq m_2$ and $i>\phi^{2}_{sp}(\mathcal{E})$. By Serre's vanishing theorem,
we can take an integer $m_1$ such that the $0$-regularity of $\mathcal{L}^{m}_1$ with respect to $\mathcal{L}_2$ is less than $-m_2-dR$ for any $m\geq m_1-d$, where $R=\max\{1, R_0\}$. By Lemma \ref{lemma3.8}, for any $1\leq j\leq d$ we obtain an exact sequence $$0\rightarrow\mathcal{E}_{d+1}\rightarrow\mathcal{E}_{d}\rightarrow\cdots\rightarrow\mathcal{E}_{0}\rightarrow\mathcal{L}^{m_1-j}_1\rightarrow0,$$ where $\mathcal{E}_{i}=V_i\otimes\mathcal{L}^{m_2+dR-iR}_2$ for $i\leq d$. Tensoring this by $\mathcal{E}^{(p^n)}$ yields an exact complex
$$\cdots\rightarrow\mathcal{E}_{1}\otimes\mathcal{E}^{(p^n)}\rightarrow\mathcal{E}_{0}\otimes\mathcal{E}^{(p^n)}
\rightarrow\mathcal{L}_1^{m_1-j}\otimes\mathcal{E}^{(p^n)}\rightarrow0.$$
By chasing through this exact complex, one obtains $$H^j(X, \mathcal{E}^{(p^n)}\otimes\mathcal{L}^{m_1-j}_1)=0$$ for any $n\geq0$ and $j>\phi^{2}_{sp}(\mathcal{E})$. This implies that $\phi^{1}_{sp}(\mathcal{E})\leq\phi^{2}_{sp}(\mathcal{E})$. Similarly, one can show that $\phi^{2}_{sp}(\mathcal{E})\leq\phi^{1}_{sp}(\mathcal{E})$. Therefore we have $\phi^{2}_{sp}(\mathcal{E})=\phi^{1}_{sp}(\mathcal{E})$. This completes the proof.
\end{proof}

\begin{lemma}\label{ext2}
Let K be an extension of $k$ and $\mathcal{E}$ a coherent sheaf on $X$. Then we have $\phi_{sp}(\mathcal{E})=\phi_{sp}(\mathcal{E}\otimes_kK)$.
\end{lemma}
\begin{proof}
The proof is similar to that of Lemma \ref{ext1}.
\end{proof}

\begin{proposition}[F-semipositivity of restrictions]\label{restriction}
Let $\mathcal{O}_X(1)$ be a very ample line bundle and $\mathcal{E}$ a coherent sheaf on $X$. Assume either that $\mathcal{E}$ is locally free or that $X$ is smooth and $k$ is infinite. Then for a general divisor $H$ in the linear system $|\mathcal{O}_X(1)|$ we have $$\mathcal{\phi}_{sp}(\mathcal{E}|_H)\leq\mathcal{\phi}_{sp}(\mathcal{E})\leq\mathcal{\phi}_{sp}(\mathcal{E}|_H)+1.$$
\end{proposition}
\begin{proof}
Using a thickening, we may assume that $p>0$. In the case that $X$ is smooth and $k$ is infinite, from \cite[Lemma 1.1.12]{HL}, one sees that for a general section $s\in H^0(X, \mathcal{O}_X(1))$ the natural morphism
$$\theta: \mathcal{E}(-1)\xrightarrow{\cdot s} \mathcal{E}$$ is injective. The smoothness of $X$ implies that $F_X$ is flat. Thus the morphism $$(F^n_X)^*\theta: \mathcal{E}^{(p^n)}(-p^n)\xrightarrow{\cdot s^{p^n}}\mathcal{E}^{(p^n)}$$ is still injective. Consider the morphism $$\tau_i: \mathcal{E}^{(p^n)}(-i)\xrightarrow{\cdot s} \mathcal{E}^{p^n}(-i+1).$$ It turns out that $(F^n_X)^*\theta$ is the composition of the $\tau_i$'s, here $1\leq i\leq p^n$. This implies that $\tau_i$ is injective for any $i$. In the case that $\mathcal{E}$ is locally free, the injectivity of $\tau_i$ is obvious. Therefore, in any case one obtains an exact sequence

\begin{equation}\label{4.1}
0\rightarrow\mathcal{E}^{(p^n)}(m-1)\rightarrow\mathcal{E}^{(p^n)}(m)\rightarrow\mathcal{E}^{(p^n)}(m)|_H\rightarrow0
\end{equation}
for any integer $m$.

By the definition of $\phi_{sp}(\mathcal{E})$, one sees that there exists an integer $m_0$ such that $$H^i(X, \mathcal{E}^{(p^n)}(m-i))=0$$
for all $i>\phi_{sp}(\mathcal{E})$, $m\geq m_0$ and $n\geq0$. On the other hand, from the short exact sequence (\ref{4.1}), we obtain the induced long exact sequence
$$H^i(X, \mathcal{E}^{(p^n)}(m-i))\rightarrow H^i(H, \mathcal{E}^{(p^n)}(m-i)|_H)\rightarrow H^{i+1}(X, \mathcal{E}^{(p^n)}(m-i-1)).$$
It follows that $H^i(H, \mathcal{E}^{(p^n)}(m-i)|_H)=0$ for $i>\phi_{sp}(\mathcal{E})$, $m\geq m_0$ and $n\geq0$. Thus we infer that $\mathcal{\phi}_{sp}(\mathcal{E}|_H)\leq\mathcal{\phi}_{sp}(\mathcal{E})$.

Similarly, from the definition of $\phi_{sp}(\mathcal{E}|_H)$, we deduce that there exists an integer $m_1$ such that $$H^i(H, \mathcal{E}^{(p^n)}(m-i)|_H)=0$$
for all $i>\phi_{sp}(\mathcal{E}|_H)$, $m\geq m_1$ and $n\geq0$. The short exact sequence (\ref{4.1}) gives the long exact sequence
\begin{eqnarray*}
&&H^i(H, \mathcal{E}^{(p^n)}(m-i)|_H)\rightarrow H^{i+1}(X, \mathcal{E}^{(p^n)}(m-i-1))\\
&&\rightarrow H^{i+1}(X, \mathcal{E}^{(p^n)}(m-i))\rightarrow H^{i+1}(H, \mathcal{E}^{(p^n)}(m-i)|_H).
\end{eqnarray*}
This implies that $$H^{i+1}(X, \mathcal{E}^{(p^n)}(m-i-1))\cong H^{i+1}(X, \mathcal{E}^{(p^n)}(m-i))$$ for any $i>\phi_{sp}(\mathcal{E}|_H)$, $m\geq m_1$ and $n\geq0$. Hence one sees that $$H^{i+1}(X, \mathcal{E}^{(p^n)}(m-i-1))\cong H^{i+1}(X, \mathcal{E}^{(p^n)}(m-i+b))$$ for any $i>\phi_{sp}(\mathcal{E}|_H)$, $m\geq m_1$, $n\geq0$ and $b\geq0$. By Serre's vanishing, one has $H^{i+1}(X, \mathcal{E}^{(p^n)}(m-i+b))=0$ for $b\gg0$. Therefore, we infer that $$H^{i+1}(X, \mathcal{E}^{(p^n)}(m-i-1))=0$$ for any $i>\phi_{sp}(\mathcal{E}|_H)$, $m\geq m_1$ and $n\geq0$.
It follows that
$\mathcal{\phi}_{sp}(\mathcal{E})\leq\mathcal{\phi}_{sp}(\mathcal{E}|_H)+1$.
\end{proof}

The F-semipositivity can be governed by Castelnuovo–Mumford regularity:
\begin{proposition}\label{est-CM}
Let $\mathcal{O}_X(1)$ be a very ample line bundle and $\mathcal{E}$ an $m$-regular sheaf on $X$. Assume either that $X$ is smooth with $\dim X=d$ or that $\mathcal{E}$ is locally free. Then we have
$$\phi_{sp}(\mathcal{E}(m+tR))\leq\max\{d-t-1, 0\},$$ where $R=\max\{1, \reg_0(\mathcal{O}_X)\}$.
\end{proposition}
\begin{proof}
By Lemma \ref{lemma3.8}, we obtain an exact sequence $$\cdots\rightarrow\mathcal{E}_{d}\rightarrow\cdots\rightarrow\mathcal{E}_{0}\rightarrow\mathcal{E}(m+tR)\rightarrow0,$$
where $\mathcal{E}_{i}=V_i\otimes\mathcal{O}_X((t-i)R)$. Pulling back it via $(F_X^n)^*$, one gets the exact sequence
$$\cdots\rightarrow V_t\otimes\mathcal{O}_X\rightarrow\cdots\rightarrow V_0\otimes\mathcal{O}_X(p^ntR)\rightarrow(\mathcal{E}(m+tR))^{(p^n)}\rightarrow0.$$
When $i\geq d-t$, one deduced that
$$h^i((\mathcal{E}(m+tR))^{(p^n)}(R))\leq\sum_{j=0}^{t}h^{i+j}(V_j\otimes\mathcal{O}_X(p^n(t-j)R+R))=0.$$
This completes the proof.
\end{proof}

The following theorem is a variant of \cite[Theorem 3]{Ara1}.

\begin{theorem}\label{tensor}
Let $\mathcal{O}_X(1)$ be a very ample line bundle on $X$. Let $\mathcal{E}$ and $\mathcal{F}$ be two coherent sheaves on $X$ such that one of them is locally free and $\mathcal{E}$ is F-ample, then
$$\phi_a(\mathcal{E}\otimes\mathcal{F})\leq\phi_{sp}(\mathcal{F}).$$
\end{theorem}
\begin{proof}
Assume that $p>0$. By the definition of $\phi_{sp}(\mathcal{F})$, one sees that there is an integer $m_0$ such that
$$H^i(X, \mathcal{F}^{(p^n)}(m-i))=0$$ for all $i>\phi_{sp}(\mathcal{F})$, $m\geq m_0$ and $n\geq0$. Let $d=\dim X$ and $R=\max\{1, \reg_0(\mathcal{O}_X)\}$.
Since $\mathcal{E}$ is F-ample, one sees that for any integer $b$ there is an $n_0$ such that $$\reg_0(\mathcal{E}^{(p^n)})\leq -m_0-dR+b$$ for all $n\geq n_0$. From Lemma \ref{lemma3.8}, we obtain an exact sequence $$\cdots\rightarrow\mathcal{E}_{d}\rightarrow\cdots\rightarrow\mathcal{E}_{0}\rightarrow\mathcal{E}^{(p^n)}\rightarrow0,$$
where $\mathcal{E}_{i}=V_i\otimes\mathcal{O}_X(m_0+dR-iR-b)$. Twisting through by $\mathcal{F}^{(p^n)}(b)$ yields an exact complex
$$\cdots\rightarrow\mathcal{E}_{1}\otimes\mathcal{F}^{(p^n)}(b)\rightarrow\mathcal{E}_{0}\otimes\mathcal{F}^{(p^n)}(b)
\rightarrow\mathcal{E}^{(p^n)}\otimes\mathcal{F}^{(p^n)}(b)\rightarrow0.$$
By chasing through this exact complex, it follows that
$$H^i(X, (\mathcal{E}\otimes\mathcal{F})^{(p^n)}\otimes\mathcal{O}_X(b))=0$$ for $n\geq n_0$ and $i>\phi_{sp}(\mathcal{F})$. Hence the conclusion follows from Lemma \ref{lemma3.3}.

In the case of $p=0$, we can carry out the above argument on the fiber of some thickening.
\end{proof}

The next result shows that the F-semipositivity of a sheaf can be considered as a limit of the F-amplitude of some sheaves.

\begin{theorem}\label{sp-a}
Suppose that $p>0$. Let $\mathcal{E}$ be a coherent sheaf on $X$, $\mathcal{L}$ an ample line bundle on $X$ and $t$ a nonnegative integer. Assume either that $X$ is smooth or that $\mathcal{E}$ is locally free. Then $\phi_{sp}(\mathcal{E})\leq t$ if and only if $\phi_{a}(\mathcal{E}^{(p^n)}\otimes \mathcal{L})\leq t$ for any $n\geq0$.
\end{theorem}
\begin{proof}
We first assume that $\phi_{sp}(\mathcal{E})\leq t$. Since $\mathcal{L}$ is an ample line bundle, one sees that $\mathcal{L}$ is F-ample by Lemma \ref{lemma3.4}. Thus Theorem \ref{tensor} gives $$\phi_{a}(\mathcal{E}^{(p^n)}\otimes \mathcal{L})\leq\phi_{sp}(\mathcal{E}^{(p^n)})=\phi_{sp}(\mathcal{E})\leq t.$$

For the other direction, we assume that $\phi_{a}(\mathcal{E}^{(p^n)}\otimes \mathcal{L})\leq t$ for any $n\geq0$. From Theorem \ref{Fujita1} and \ref{Fujita2}, it follows that there exists an integer $m_0(\mathcal{L})$ such that $$H^i(X, \mathcal{E}^{(p^n)}\otimes \mathcal{L}^m)=0$$ for any $i>t$, $m>m_0(\mathcal{L})$ and $n\geq0$. Since $\mathcal{L}^m$ is very ample for $m\gg0$, one concludes that $\phi_{sp}(\mathcal{E})\leq t$.
\end{proof}

By Theorem \ref{sp-a}, we can translate \cite[Theorem 1]{Ara1} to F-semipositivity.
\begin{theorem}\label{Thm4.10}
Let $\mathcal{E}, \mathcal{E}_i, \mathcal{E}^j$ be coherent sheaves on $X$, where $1\leq i, j\leq m$. Assume either that $X$ is smooth or that these sheaves are locally free. Then the following statements hold.
\begin{enumerate}
  \item Given an exact sequence $0\rightarrow\mathcal{E}_1\rightarrow\mathcal{E}\rightarrow\mathcal{E}_2\rightarrow0$, we have $\phi_{sp}(\mathcal{E})\leq\max\{\phi_{sp}(\mathcal{E}_1), \phi_{sp}(\mathcal{E}_2)\}$.
  \item Let $0\rightarrow\mathcal{E}_m\rightarrow\cdots\rightarrow\mathcal{E}_0\rightarrow\mathcal{E}\rightarrow0$ be an exact sequence such that $\phi_{sp}(\mathcal{E}_i)\leq i+l$ for each $i$, then $\phi_{sp}(\mathcal{E})\leq l$.
  \item Let $0\rightarrow\mathcal{E}\rightarrow\mathcal{E}^0\rightarrow\cdots\rightarrow\mathcal{E}^m\rightarrow0$ be an exact sequence such that $\phi_{sp}(\mathcal{E}^i)\leq l-i$ for each $i$, then $\phi_{sp}(\mathcal{E})\leq l$.
  \item Let $f: Y\rightarrow X$ be a proper morphism of projective varieties such that $d$ is the maximum dimension of the closed fibers. If $\mathcal{E}$ is locally free then $\phi_{sp}(f^*\mathcal{E})\leq \phi_{sp}(\mathcal{E})+d$.
  \item If $f: Y\rightarrow X$ is an \'etale morphism of smooth projective varieties, then $\phi_{sp}(f_*\mathcal{E})= \phi_{sp}(\mathcal{E})$.
\end{enumerate}
\end{theorem}
\begin{proof}
Let $\mathcal{L}$ be an ample line bundle on $X$. Using a thickening, we can assume that $p>0$. Applying \cite[Theorem 1]{Ara1} to $\mathcal{E}^{(p^n)}\otimes \mathcal{L}$, one sees that the conclusions follow from Theorem \ref{sp-a}.
\end{proof}

\begin{proposition}\label{gg}
Assume that $X$ is integral with $\dim X>0$. Let $\mathcal{E}$ be a globally generated locally free sheaf on $X$ of rank $r$. Then $\phi_{sp}(\mathcal{E})<\dim X$. Furthermore, when $p=0$ we have $\phi_{sp}(\mathcal{E})<r$.
\end{proposition}
\begin{proof}
We first assume that $p=0$. Let $\mathcal{L}$ be an ample line bundle on $X$. Let $(\widetilde{X}\rightarrow \Spec A, \widetilde{\mathcal{E}}, \widetilde{\mathcal{L}})$ be a thickening of $(X, \mathcal{E}, \mathcal{L})$. We can shrink $\Spec A$ so that $\mathcal{L}_q$ is ample and $\mathcal{E}_q$ is generated by global sections for any $q\in\Spec A$. Let $p_q=\chr A/q$. Then one sees that $\mathcal{E}_q^{(p_q^n)}\otimes\mathcal{L}$ is ample for any $n\geq0$. Hence by Theorem \ref{ample}, one deduces that $\phi_a(\mathcal{E}_q^{(p_q^n)}\otimes\mathcal{L})<\min\{r, \dim X\}$. It follows from Theorem \ref{sp-a} that $\phi_{sp}(\mathcal{E}_q)<\min\{r, \dim X\}$ for any $q\in\Spec A$. Therefore we obtain the conclusions in characteristic zero.

The proof of $\phi_{sp}(\mathcal{E})<\dim X$ in the case of $p>0$ is similar.
\end{proof}

We can get rid of the local freeness assumption in the proposition above if $X$ is smooth.
\begin{proposition}\label{gg2}
Assume that $X$ is smooth with $\dim X>0$. Let $\mathcal{E}$ be a globally generated coherent sheaf on $X$. Then $\phi_{sp}(\mathcal{E})<\dim X$.
\end{proposition}
\begin{proof}
Since the F-semipositivity is invariant under field extensions, we may assume that $k$ is algebraically closed. If $\dim X=1$, one sees that $\mathcal{E}$ is a direct sum of the torsion free part of $\mathcal{E}$ with the torsion part. Since the torsion sheaves on a curve have vanishing  cohomology of positive degree, we may assume that $\mathcal{E}$ is locally free. Then one gets $\phi_{sp}(\mathcal{E})<1$ by Proposition \ref{gg}. If $\dim X>1$, then induction on $\dim X$ and Proposition \ref{restriction} yields the conclusion.
\end{proof}

\subsection{Asymptotic Riemann-Roch theorem}

By Theorem \ref{tensor}, one sees that the Fujita type vanishing theorems (Theorem \ref{Fujita1} and \ref{Fujita2}) still hold for F-semipositivity.

\begin{theorem}\label{Fujita3}
Let $\mathcal{L}$ be an ample line bundle on $X$. Given any coherent sheaf $\mathcal{F}$ on $X$, there exists an integer $m(\mathcal{F}, \mathcal{L})$ such that
$$H^i(X, \mathcal{F}\otimes\mathcal{L}^m\otimes\mathcal{G})=0$$ for all $i>t$, $m\geq m(\mathcal{F}, \mathcal{L})$ and any locally free sheaf $\mathcal{G}$ on $X$ with $\phi_{sp}(\mathcal{G})\leq t$.
\end{theorem}

\begin{theorem}\label{Fujita4}
Assume that $X$ is smooth. Let $\mathcal{L}$ be an ample line bundle on $X$. Given any locally free sheaf $\mathcal{F}$ on $X$, there exists an integer $m(\mathcal{F}, \mathcal{L})$ such that
$$H^i(X, \mathcal{F}\otimes\mathcal{L}^m\otimes\mathcal{G})=0$$ for all $i>t$, $m\geq m(\mathcal{F}, \mathcal{L})$ and any coherent sheaf $\mathcal{G}$ on $X$ with $\phi_{sp}(\mathcal{G})\leq t$.
\end{theorem}

\begin{proposition}[Growth of cohomology]\label{prop4.11}
Suppose that $p>0$. Let $\mathcal{F}$ and $\mathcal{G}$ be coherent sheaves on $X$. We assume either that $X$ is smooth and $\mathcal{F}$ is locally free or that $\mathcal{G}$ is locally free. Then we have $$h^i(X,\mathcal{F}\otimes\mathcal{G}^{(p^n)})=O(p^{nd})$$ for every $i\geq0$, where $d=\dim X$.
\end{proposition}
\begin{proof}
Let $\mathcal{O}_X(1)$ be a very ample line bundle on $X$. Let $R=\max\{1, \reg_0(\mathcal{O}_X)\}$ and $m=\reg_0(\mathcal{G})$. By Lemma \ref{lemma3.8}, we obtain an exact sequence $$\cdots\rightarrow\mathcal{G}_{d}\rightarrow\cdots\rightarrow\mathcal{G}_{0}\rightarrow\mathcal{G}\rightarrow0,$$
where $d=\dim X$ and $\mathcal{G}_{i}=V_i\otimes\mathcal{O}_X(-m-iR)$ for $i\leq d$. Pulling back it via $(F_X^n)^*$ and then tensoring this by $\mathcal{F}$, by our assumptions one gets the exact sequence
$$\cdots\rightarrow\mathcal{F}\otimes V_1\otimes\mathcal{O}_X(-mp^n-Rp^n)\rightarrow \mathcal{F}\otimes V_0\otimes\mathcal{O}_X(-mp^n)\rightarrow\mathcal{F}\otimes\mathcal{G}^{(p^n)}\rightarrow0.$$
Therefore, one obtains
$$
h^i(\mathcal{F}\otimes\mathcal{G}^{(p^n)})\leq\sum_{j=0}^{d}h^{i+j}(\mathcal{F}\otimes V_j\otimes\mathcal{O}_X(-mp^n-jRp^n)).
$$
The desired conclusion follows from \cite[Example 1.2.33]{Laz}.
\end{proof}

We conclude with a result that strengthens the statement above.
\begin{proposition}\label{prop4.12}
Suppose that $p>0$. Let $\mathcal{F}$ and $\mathcal{G}$ be coherent sheaves on $X$. We assume either that $X$ is smooth and $\mathcal{F}$ is locally free or that $\mathcal{G}$ is locally free. Then we have $$h^i(X,\mathcal{F}\otimes\mathcal{G}^{(p^n)})=O(p^{nd-n})$$ for $i>\phi_{sp}(\mathcal{G})$, where $d=\dim X$.
\end{proposition}
\begin{proof}
Since the F-semipositivity and the dimension of the cohomology are invariant under field extensions, we may assume that $k$ is infinite.
Thanks to Theorem \ref{Fujita3} and \ref{Fujita4}, there exists an effective very ample Cartier divisor $H$ having the property that
$$H^i(X, \mathcal{F}\otimes\mathcal{G}^{(p^n)}\otimes\mathcal{O}_X(H))=0$$ for all $i>\phi_{sp}(\mathcal{G})$ and $n\geq0$. By the same argument as in the proof of Proposition \ref{restriction}, we have the exact sequence
$$0\rightarrow\mathcal{F}\otimes\mathcal{G}^{(p^n)}\rightarrow\mathcal{F}\otimes\mathcal{G}^{(p^n)}(H)
\rightarrow\mathcal{F}\otimes\mathcal{G}^{(p^n)}(H)\otimes\mathcal{O}_H\rightarrow0.$$
Therefore, one obtains $$h^i(X, \mathcal{F}\otimes\mathcal{G}^{(p^n)})\leq h^{i-1}(H, \mathcal{F}\otimes\mathcal{G}^{(p^n)}(H)\otimes\mathcal{O}_H)$$ for $i>\phi_{sp}(\mathcal{G})$.
Since $$h^{i-1}(H, \mathcal{F}\otimes\mathcal{G}^{(p^n)}(H)\otimes\mathcal{O}_H)=O(p^{nd-n})$$ by Proposition \ref{prop4.11}, we complete the proof.
\end{proof}

\begin{corollary}[Asymptotic Riemann-Roch]\label{cor4.13}
Suppose that $X$ is smooth with $\dim X=d$ and $p>0$. Let $\mathcal{E}$ be a coherent sheaf on $X$ with $\phi_{sp}(\mathcal{E})\leq1$. One has
$$h^0(X, \mathcal{E}^{(p^n)})-h^1(X, \mathcal{E}^{(p^n)})=\ch_d (\mathcal{E})p^{nd}+O(p^{n(d-1)}).$$
\end{corollary}
\begin{proof}
By Riemann-Roch thorem, one sees that
\begin{eqnarray*}
\chi(X, \mathcal{E}^{(p^n)}) &=& \ch(\mathcal{E}^{(p^n)})\td(X)\\
&=&\ch_d (\mathcal{E})p^{nd}-\frac{1}{2}K_X\ch_{d-1} (\mathcal{E})p^{n(d-1)}+O(p^{n(d-2)}).
\end{eqnarray*}
Hence the desired formula follows from Proposition \ref{prop4.12}.
\end{proof}

\section{Bogomolov type inequalities}\label{S5}
The aim of this section is to prove Theorem \ref{Bog1}, Corollary \ref{Bog4} and Theorem \ref{Bog2}. Let $k$ be an algebraically closed field with $\chr(k)=p$. Let $X$ be a smooth $d$-dimensional projective variety defined over $k$ with an ample divisor $H$.

We give a more general version of Theorem \ref{Bog1}:
\begin{theorem}\label{Thm5.1}
Assume that $p=0$. Let $\mathcal{E}$ be a rank $r$ torsion free sheaf on $X$ with $\phi_{sp}(\mathcal{E})\leq1$. For any $1\leq t\leq d$, we have
$$H^{d-t}\ch_t(\mathcal{E})\leq\max\left\{\frac{r(\mu_H^+(\mathcal{E}))^t}{t!(H^d)^{t-1}}, 0\right\}.$$
\end{theorem}
\begin{proof}
Since the conclusion is invariant if we replace $H$ by $mH$ for any $m\geq1$, we can assume that $H$ is sufficiently ample so that one can take $$H_1, \cdots H_{d-t}\in |H|$$ satisfying that $X_t:=H_1\cap\cdots\cap H_{d-t}$ is a smooth $t$-dimensional subvariety. By the restriction
theorem (e.g. \cite[Theorem 5.2]{Langer1}) and Proposition \ref{restriction}, one can also assume that $\mathcal{E}|_{X_t}$ is torsion free, $\mu_H^+(\mathcal{E}|_{X_t})=\mu_H^+(\mathcal{E})$ and $\phi_{sp}(\mathcal{E}|_{X_t})\leq1$.

Without loss of generality, we may assume that $t=d$, and thus $X_t=X$. Let $(\widetilde{X}\rightarrow\Spec A, \widetilde{\mathcal{E}}, \widetilde{H})$ be a thickening of $(X, \mathcal{E}, H)$. We freely use the notations in Section \ref{S3}. One can shrink $\Spec A$ so that $H_q$ is very ample for any $q\in\Spec A$. By the openness of semistability, one sees that $\mu_{H_q}^+(\mathcal{E}_q)=\mu_H^+(\mathcal{E})$ for a general $q\in \Spec A$. Let $\chr(A/q)=p_q$, by Corollary \ref{cor4.13}, one has
\begin{equation}\label{5.1}
\ch_d (\mathcal{E}_q)p_q^{nd}+O(p_q^{n(d-1)})\leq h^0(X_q, \mathcal{E}_q^{(p_q^n)}).
\end{equation}
Let $B$ be a nef divisor on $X$ such that $T_X(B)$ is globally generated. Then $T_{X_q}(B_q)$ is aslo globally generated for a general $q\in \Spec A$. From Theorem \ref{Langer1} and \ref{Langer2} it follows that
\begin{eqnarray*}
h^0(X_q, \mathcal{E}_q^{(p_q^n)})&\leq&\max\left\{
rH_q^d\binom{\mu_{H_q}^+(\mathcal{E}_q^{(p_q^n)})/H_q^d+d+f(r)}{d},0\right\}\\
&\leq&\max\left\{rH_q^d\binom{p_q^n\Big(\mu_{H_q}^+(\mathcal{E}_q)+\frac{r-1}{p_q-1}B_qH_q^{d-1}\Big)/H_q^d+d+f(r)}{d},0\right\}\\
&\leq&\max\left\{\frac{rp_q^{nd}}{d!(H_q^d)^{d-1}}\Big(\mu_{H_q}^+(\mathcal{E}_q)+\frac{r-1}{p_q-1}B_qH_q^{d-1}\Big)^d+O(p_q^{n(d-1)}),0\right\}
\end{eqnarray*}
This and (\ref{5.1}) imply that
$$\ch_d (\mathcal{E}_q)p_q^{nd}\leq\max\left\{\frac{rp_q^{nd}}{d!(H_q^d)^{d-1}}\Big(\mu_{H_q}^+(\mathcal{E}_q)+\frac{r-1}{p_q-1}B_qH_q^{d-1}\Big)^d,
0\right\}+O(p_q^{n(d-1)}).$$
Taking $n\rightarrow+\infty$, one infers
\begin{equation}\label{5.2}
\ch_d (\mathcal{E}_q)\leq\max\left\{\frac{r}{d!(H_q^d)^{d-1}}\Big(\mu_{H_q}^+(\mathcal{E}_q)+\frac{r-1}{p_q-1}B_qH_q^{d-1}\Big)^d,
0\right\}
\end{equation}
The desired inequality follows as $p_q$ approaches infinity.
\end{proof}

By the proof above, we can obtain some similar inequalities in positive characteristic.
\begin{theorem}\label{Bog3}
Assume that $p>0$. Let $B$ be a nef divisor on $X$ such that $T_X(B)$ is globally generated. Let $\mathcal{E}$ be a rank $r$ torsion free sheaf on $X$ with $\phi_{sp}(\mathcal{E})\leq1$. Then we have
$$\ch_d(\mathcal{E})\leq\max\left\{\frac{r\Big(\mu_H^+(\mathcal{E})+\frac{r-1}{p-1}BH^{d-1}\Big)^d}{d!(H^d)^{d-1}}, 0\right\}.$$ Moreover if $\mathcal{E}$ is strongly $\mu_H$-semistable, then $$\ch_d(\mathcal{E})\leq\max\left\{\frac{r(\mu_H(\mathcal{E}))^d}{d!(H^d)^{d-1}}, 0\right\}.$$
\end{theorem}
\begin{proof}
One notices that these two inequalities are also invariant if we replace $H$ by $mH$ for any $m\geq1$, we can assume that $H$ is very ample.
The proof of the first inequality is the same as that of (\ref{5.2}). When $\mathcal{E}$ is strongly semistable, applying the first inequality to $\mathcal{E}^{(p^n)}$, one gets
$$p^{nd}\ch_d(\mathcal{E})\leq\max\left\{\frac{r\Big(p^{n}\mu_H(\mathcal{E})+\frac{r-1}{p-1}BH^{d-1}\Big)^d}{d!(H^d)^{d-1}}, 0\right\}.$$
Taking $n\rightarrow+\infty$, one obtains the second inequality.
\end{proof}

\begin{proof}[Proof of Corollary \ref{Bog4}]
By Proposition \ref{est-CM}, one sees that $\phi_{sp}(\mathcal{E}(mH))\leq1$ for $m\gg0$. From Theorem \ref{Thm5.1} and \ref{Bog3}, it follows that $$d!(H^dr)^{d-1}\ch_d(\mathcal{E}(mH))\leq \Big(H^{d-1}\ch_1(\mathcal{E}(mH))\Big)^d.$$ Expanding it, one obtains
$$\sum_{i=2}^d\binom{d}{i}(H^dr)^{d-i}\Big((H^{d-1}\ch_1(\mathcal{E}))^i-i!(H^dr)^{i-1}H^{d-i}\ch_i(\mathcal{E})\Big)m^{d-i}\geq0$$ for $m\gg0$.
This implies the desired conclusions.
\end{proof}

Now we show Theorem \ref{Bog2}. Let $(A, \Theta)$ be a polarized complex abelian variety of dimension $d$. Consider $\widehat{A}:= \Pic^0(A)$ and $\mathcal{P}$ a Poincar\'e line bundle on $A\times\widehat{A}$. This gives as usual two Fourier-Mukai functors
\begin{eqnarray*}
 & R\Phi_{\mathcal{P}}: D^b(A)\rightarrow D^b(\widehat{A}),& \mathcal{F}\mapsto Rp_{\widehat{A}*}(p_A^*(\mathcal{F})\otimes\mathcal{P}); \\
 & R\Psi_{\mathcal{P}}: D^b(\widehat{A})\rightarrow D^b(A),& \mathcal{G}\mapsto Rp_{A*}(p_{\widehat{A}}^*(\mathcal{G})\otimes\mathcal{P}).
\end{eqnarray*}
For $b\in \mathbb{Z}$, $\underline{b}: A\rightarrow A$, $z\mapsto bz$ denotes the multiplication homomorphism.

We review the notion of $GV$-object introduced by Pareschi and Popa in \cite{PP}.
\begin{definition}
For any integer $m\geq0$, an object $\mathcal{F}\in D^b(A)$ is called a $GV_{-m}$ object if $$\codim \Supp(R^i\Phi_{\mathcal{P}}\mathcal{F})\geq i-m$$ for all $i\geq0$. The $i$-th cohomological support locus of $\mathcal{F}$ is
$$V^i(\mathcal{F}):=\{\alpha\in\widehat{A}~|~h^i(A, \mathcal{F}\otimes\alpha)>0\}.$$ Define $h_{gen}^i(A, \mathcal{F})$ as the dimension of hypercohomology $H^i(A, \mathcal{F}\otimes\alpha)$ for $\alpha$ general in $\widehat{A}$.
\end{definition}

\begin{lemma}\label{lemma5.4}
The following conditions are equivalent:
\begin{enumerate}
  \item $\mathcal{F}$ is a $GV_{-m}$ object.
  \item $\codim V^i(\mathcal{F})\geq i-m$ for all $i$.
\end{enumerate}
\end{lemma}
\begin{proof}
See \cite[Lemma 3.6]{PP}.
\end{proof}

\begin{proof}[Proof of Theorem \ref{Bog2}]
As the proof of Theorem \ref{Thm5.1}, we can assume that $\Theta$ is very ample.
Since $\mathcal{E}$ is a $GV_{-1}$ sheaf, by Lemma \ref{lemma5.4} one deduces that $h_{gen}^i(A, \mathcal{E})=0$ for $i\geq2$. Thus for any $b\in\mathbb{Z}$ and $i\geq2$, we have $$h_{gen}^i(A, \underline{b}^*\mathcal{E})=\sum_{\alpha\in \widehat{\underline{b}}^{-1}(\widehat{e})}h_{gen}^i(A, \mathcal{E}\otimes\alpha)=0,$$ where $\widehat{e}$ is the identity point of $\widehat{A}$ and $\widehat{\underline{b}}$ is the dual isogeny.
From Riemann-Roch theorem, it follows that $$h_{gen}^0(A, \underline{b}^*\mathcal{E})\geq\chi(\underline{b}^*\mathcal{E})=\ch_d(\underline{b}^*\mathcal{E})=b^{2d}\ch_d(\mathcal{E}).$$
On the other hand, Theorem \ref{Langer1} gives
\begin{eqnarray*}
h_{gen}^0(A, \underline{b}^*\mathcal{E})&\leq&\max\left\{
r{\Theta}^d\binom{\mu_{\Theta}^+(\underline{b}^*\mathcal{E})/{\Theta}^d+d+f(r)}{d},0\right\}\\
   &=& \max\left\{
r{\Theta}^d\binom{b^2\mu_{\Theta}^+(\mathcal{E})/{\Theta}^d+d+f(r)}{d},0\right\}\\
&\leq&\max\left\{\frac{rb^{2d}}{d!({\Theta}^d)^{d-1}}(\mu_{\Theta}^+(\mathcal{E}))^d,0\right\}+O(b^{2d-2}).
\end{eqnarray*}
Hence one obtains
$$b^{2d}\ch_d(\mathcal{E})\leq\max\left\{\frac{rb^{2d}}{d!({\Theta}^d)^{d-1}}(\mu_{\Theta}^+(\mathcal{E}))^d,0\right\}+O(b^{2d-2}).$$
Taking $b\rightarrow+\infty$, we conclude that $$\ch_d(\mathcal{E})\leq\max\left\{\frac{r(\mu_{\Theta}^+(\mathcal{E}))^d}{d!({\Theta}^d)^{d-1}},0\right\}.$$
\end{proof}

\section{Chern classes of varieties}\label{S6}
In this section we exhibit some applications of Theorem \ref{Bog1}
to the Chern classes of threefolds and varieties of small codimension in projective spaces and abelian varieties.

Let $X$ be a smooth variety of dimension $d$ in projective space $\mathbb{P}^{d+e}$ over complex number field.
The cotangent bundle of $X$ is determined by the Euler sequence of $\mathbb{P}^{d+e}$ restricted to $X$
\begin{equation}\label{Euler}
0\rightarrow\Omega^1_{\mathbb{P}^{d+e}}|_X\rightarrow(\mathcal{O}_X(-1))^{\oplus(d+e+1)}\rightarrow\mathcal{O}_X\rightarrow0
\end{equation}
and the conormal bundle exact sequence
\begin{equation}\label{6.2}
0\rightarrow\mathcal{N}^{*}\rightarrow\Omega^1_{\mathbb{P}^{d+e}}|_X\rightarrow\Omega^1_{X}\rightarrow0.
\end{equation}
From these two exact sequences, it follows that $\mathcal{N}(-1)$ is generated by global sections. So is $$\wedge^{e-1}(\mathcal{N}(-1))\cong(\wedge^{e-1}\mathcal{N})(1-e).$$ Hence by Proposition \ref{gg}, one infers that
$$
\phi_{sp}((\wedge^{e-1}\mathcal{N})(1-e))\leq\min\{\rank(\wedge^{e-1}\mathcal{N})-1, d-1\}=\min\{e-1, d-1\}.
$$
Since $$\wedge^{e-1}\mathcal{N}\cong\mathcal{N}^{*}\otimes\det\mathcal{N}\cong\mathcal{N}^{*}\otimes\omega_X(d+e+1),$$
one obtains $$\phi_{sp}(\mathcal{N}^{*}\otimes\omega_X(d+2))\leq \min\{e-1, d-1\}.$$ Twisting (\ref{6.2}) by $\omega_X(d+2)$ yields an exact sequence
$$0\rightarrow\mathcal{N}^{*}\otimes\omega_X(d+2)\rightarrow\Omega^1_{\mathbb{P}^{d+e}}|_X\otimes\omega_X(d+2)
\rightarrow\Omega^1_{X}\otimes\omega_X(d+2)\rightarrow0.$$
On the other hand, by (\ref{Euler}) and Theorem \ref{Thm4.10}, one deduces that $$\phi_{sp}(\Omega^1_{\mathbb{P}^{d+e}}(1)|_X\otimes\mathcal{L})\leq1$$ for any F-semipositive line bundle $\mathcal{L}$ on $X$. Therefore, by Theorem \ref{Thm4.10} again, one sees that
\begin{equation}\label{6.3}
\phi_{sp}(\Omega^1_{X}\otimes\omega_X(d+2))\leq 1
\end{equation}
if $\omega_X(d+1)$ is F-semipositive and $1\leq\min\{e, d\}\leq3$.  Hence Theorem \ref{Thm5.1} gives:
\begin{theorem}\label{Chern1}
Assume that $1\leq\min\{e, d\}\leq3$. Let $H_X$ be the restriction of the hyperplane of $\mathbb{P}^{d+e}$ to $X$. Let $t$ be an integer satisfying $1\leq t\leq d$ and $L$ an ample divisor on $X$.
If $\omega_X(d+1)$ is F-semipositive, then
$$L^{d-t}\ch_t\Big(\Omega_X^1\otimes\omega_X(d+2)\Big)\leq
\max\Big\{\frac{d\big(\mu_L^+(\Omega_X^1\otimes\omega_X(d+2))\big)^t}{t!(L^d)^{t-1}},0\Big\}.$$
\end{theorem}

It is well known that $\Omega^1_X$ is
$\mu_{K_X}$-stable when $\omega_X$ is ample and $\mu_{\mathcal{O}_X(1)}$-stable when $\omega_X$ is trivial (see \cite{En} for example). Thus we have the following corollary.
\begin{corollary}\label{Chern0}
Assume that $1\leq\min\{e, d\}\leq3$. Let $H_X$ be the restriction of the hyperplane of $\mathbb{P}^{d+e}$ to $X$. Let $t$ be an integer satisfying $1\leq t\leq d$.
If $\omega_X$ is ample, then
$$K_X^{d-t}\ch_t\Big(\Omega_X^1\otimes\omega_X(d+2)\Big)\leq\frac{\Big((d+1)K_X^d+(d^2+2d)K_X^{d-1}H_X\Big)^t}{t!(dK_X^d)^{t-1}}.$$
If $\omega_X$ is trivial, then
$$H_X^{d-t}\ch_t(\Omega_X^1(d+2))\leq\frac{(d^2+2d)^t}{t!d^{t-1}}H_X^d.$$
\end{corollary}

Applying the theorem to case that $d=t=3$, one obtains the following inequalities for threefolds.
\begin{corollary}\label{cor6.2}
Let $X$ be a complex projective smooth threefold. Let $H$ be a very ample divisor on $X$. If $K_X$ is ample, then $$\ch_3\Big(\Omega_X^1\otimes\omega_X(5H)\Big)\leq\frac{\Big(4K_X^3+15K_X^{2}H\Big)^3}{54(K_X^3)^2}.$$
If $K_X=0$, then one has
$$c_3(X)+10Hc_2(X)\geq0.$$
\end{corollary}
\begin{proof}
It is obvious that the first inequality follows from Corollary \ref{Chern0}.

We now assume that $K_X=0$. Hence Theorem \ref{Chern1} gives $$\ch_3(\Omega_X^1(5H))\leq\frac{15^3}{54}H^3.$$
This implies $$\ch_3(\Omega_X^1)+5H\ch_2(\Omega_X^1)\leq0.$$
Since $$\ch_3(\Omega_X^1)=-\ch_3(T_X)=-\frac{1}{6}(c_1^3(X)-3c_1(X)c_2(X)+3c_3(X))=-\frac{1}{2}c_3(X)$$ and
$$\ch_2(\Omega_X^1)=\ch_2(T_X)=-c_2(X),$$ the conclusion follows.
\end{proof}

Theorem \ref{Chern1} can be applied to Fano threefolds.
\begin{corollary}\label{cor-Fano}
Let $X$ be a complex smooth Fano threefold and $l_X$ the Fano index of $X$. Assume that $-K_X\sim l_XL$ for an ample divisor $L$ in $X$. Then one has
$$(\frac{13}{2}l_X-25)Lc_2(X)-\frac{1}{2}c_3(X)\leq(\frac{127}{6}l_X^3-\frac{325}{2}l_X^2+\frac{625}{2}l_X)L^3.$$ If $\Omega_X^1$ is $\mu_L$-semistable and $l_X\leq3$, we have
$$(\frac{13}{2}l_X-25)Lc_2(X)-\frac{1}{2}c_3(X)\leq(\frac{58}{27}l_X^3-\frac{25}{3}l_X^2)L^3.$$
\end{corollary}
\begin{proof}
It is well known that $1\leq l_X\leq 4$, and $l_X=4$ if and only if $X$ is isomorphic to $\mathbb{P}^3$. By \cite{Lee}, one sees that $$H_X:=K_X+5L=(5-l_X)L$$ is very ample. Since  $$\omega_X(4H_X)=\mathcal{O}_X((20-5l_X)L)$$ is F-semipositivity, Theorem \ref{Chern1} gives 
\begin{equation}\label{7}
\ch_3(\Omega_X^1((25-6l_X)L))\leq\max\Big\{\frac{\big(\mu_L^+(\Omega_X^1((25-6l_X)L))\big)^3}{2(L^3)^{2}},0\Big\}.
\end{equation}
By \cite[Theorem 1.3]{Peter}, one deduces that $\mu_L^+(\Omega_X^1)<0$. It follows that $$\ch_3(\Omega_X^1((25-6l_X)L))<\frac{(25-6l_X)^3L^3}{2}.$$ This implies 
$$\ch_3(\Omega_X^1)+(25-6l_X)L\ch_2(\Omega_X^1)<\frac{l_X(25-6l_X)^2}{2}L^3.$$
A straightforward calculation gives the first inequality.

If $\Omega_X^1$ is $\mu_L$-semistable and $l_X\leq3$, the inequality (\ref{7}) implies $$\ch_3(\Omega_X^1)+(25-6l_X)L\ch_2(\Omega_X^1)
\leq(\frac{25}{6}l_X^2-\frac{55}{54}l_X^3)L^3.$$
Expanding it, we get the second inequality.
\end{proof}

One can obtain a similar result for varieties in abelian varieties.
\begin{theorem}\label{Chern2}
Let $A$ be a complex abelian variety of dimension $d+e$ and $Y$ be a smooth subvariety in $A$ of dimension $d$. Assume that $1\leq\min\{e, d\}\leq3$, and let $t$ be an integer satisfying $1\leq t\leq d$.
If $\omega_Y$ is ample, then
$$K_Y^{d-t}\ch_t(\Omega_Y^1\otimes\omega_Y)\leq\frac{(d+1)^t}{t!d^{t-1}}K_Y^d.$$
\end{theorem}
\begin{proof}
Since the tangent bundle of $A$ is trivial, one has the conormal bundle exact sequence
\begin{equation}\label{6.4}
0\rightarrow\mathcal{M}^{*}\rightarrow\mathcal{O}_Y^{\oplus(d+e)}\rightarrow\Omega^1_{Y}\rightarrow0.
\end{equation}
It turns out that $\mathcal{M}$ is generated by global sections. So is
$$\wedge^{e-1}\mathcal{M}\cong\mathcal{M}^{*}\otimes\det\mathcal{M}\cong\mathcal{M}^{*}\otimes\omega_Y.$$
By Proposition \ref{gg}, one sees that $$\phi_{sp}(\mathcal{M}^{*}\otimes\omega_Y)\leq \min\{e-1, d-1\}\leq2.$$
Tensoring (\ref{6.4}) by $\omega_Y$ yields an exact sequence
$$0\rightarrow\mathcal{M}^{*}\otimes\omega_Y\rightarrow\omega_Y^{\oplus(d+e)}\rightarrow\Omega^1_{Y}\otimes\omega_Y\rightarrow0.$$
By Theorem \ref{Thm4.10} one obtains $\phi_{sp}(\Omega^1_{Y}\otimes\omega_Y)\leq1$ if $\omega_X$ is ample.
Thus Theorem \ref{Bog1} gives the desired conclusions.
\end{proof}

From the proof of Theorem \ref{Chern2}, one can also obtain a more general result:
\begin{theorem}\label{Thm6.5}
Let $(X, H)$ be a polarized complex projective smooth threefold. Let $\mathcal{E}$ be a torsion free sheaf of rank $r$ on $X$. If $\det\mathcal{E}$ is ample or trivial and $\mathcal{E}$ is generically globally generated in dimension one, then one has $$\ch_3(\mathcal{E}\otimes\det\mathcal{E})\leq\frac{r(\mu_H^+(\mathcal{E}\otimes\det\mathcal{E}))^3}{6(H^3)^2}.$$
\end{theorem}
\begin{proof}
By Theorem \ref{Thm5.1}, we only need to show $\phi_{sp}(\mathcal{E}\otimes\det\mathcal{E})\leq1$.

From our assumption on $\mathcal{E}$, one has an exact sequence
\begin{equation}\label{6.5}
0\rightarrow\mathcal{M}\rightarrow\mathcal{O}_X^{\oplus h^0(\mathcal{E})}\xrightarrow{ev_{\mathcal{E}}}\mathcal{E}\rightarrow\mathcal{Q}\rightarrow0,
\end{equation}
where $\dim \mathcal{Q}\leq1$. It follows that $\det\mathcal{E}\cong\det (\im ev_{\mathcal{E}})$ and
\begin{eqnarray*}
 \phi_{sp}(\mathcal{E}\otimes\det\mathcal{E})  &\leq& \max\{\phi_{sp}(\mathcal{Q}\otimes\det\mathcal{E}), \phi_{sp}((\im ev_{\mathcal{E}})\otimes\det\mathcal{E})\} \\
   &\leq& \max\{1, \phi_{sp}((\im ev_{\mathcal{E}})\otimes\det (\im ev_{\mathcal{E}}))\}.
\end{eqnarray*}
Thus one completes the proof if one can show $\phi_{sp}((\im ev_{\mathcal{E}})\otimes\det(\im ev_{\mathcal{E}}))\leq1$. Without loss of generality, we can assume that $\mathcal{Q}=0$ and $\mathcal{E}$ is generated by global sections, and the sequence (\ref{6.5}) becomes
\begin{equation*}
0\rightarrow\mathcal{M}\rightarrow\mathcal{O}_X^{\oplus h^0(\mathcal{E})}\xrightarrow{ev_{\mathcal{E}}}\mathcal{E}\rightarrow0.
\end{equation*}
Applying the functor $\mathcal{H}om(-, \mathcal{O}_X)$ to this, one obtains a long exact sequence
$$0\rightarrow\mathcal{E}^*\rightarrow\mathcal{O}_X^{\oplus h^0(\mathcal{E})}\xrightarrow{g}\mathcal{M}^*\rightarrow\mathcal{E}xt^1(\mathcal{E}, \mathcal{O}_X)\rightarrow0.$$
Consider the morphism $$h: \wedge^{r_0-1}(\im g)\rightarrow \wedge^{r_0-1}(\mathcal{M}^*)$$ induced by the inclusion $\im g\hookrightarrow\mathcal{M}^*$. By the torsion freeness of $\mathcal{E}$, $\im g$ and $\mathcal{M}^*$, one deduces that $\dim\mathcal{E}xt^1(\mathcal{E}, \mathcal{O}_X)\leq1$, $\dim\ker h\leq1$ and $\dim\coker h\leq1$. This implies
\begin{equation}\label{6.6}
 \phi_{sp}(\coker h)\leq1.
\end{equation}
Since $\im g$ is generated by global sections, the same thing holds for $\wedge^{r_0-1}(\im g)$ and $\wedge^{r_0-1}(\im g)/\ker h$. By Proposition \ref{gg2}, one sees that
\begin{equation}\label{6.7}
\phi_{sp}(\wedge^{r_0-1}(\im g)/\ker h)\leq2.
\end{equation}
From (\ref{6.6}), (\ref{6.7}) and the exact sequence
$$0\rightarrow\wedge^{r_0-1}(\im g)/\ker h\rightarrow\wedge^{r_0-1}(\mathcal{M}^*)\rightarrow\coker h\rightarrow0,$$ it follows that
\begin{equation}\label{6.8}
\phi_{sp}(\wedge^{r_0-1}(\mathcal{M}^*))\leq2.
\end{equation}

On the other hand, by \cite[Proposition 1.1]{Hart2}, one sees that $\mathcal{M}$ is reflexive. So is $\mathcal{M}^*$. From \cite[Corollary 1.4]{Hart2}, one can find an open subset $U\subset X$ such that both $\mathcal{M}|_U$ and $\mathcal{M}^*|_U$ are locally free and $\dim(X-U)=0$. Let $$r_0=\rank(\mathcal{M})=h^0(\mathcal{E})-r.$$
The composition of the multiplication morphism $$\mathcal{M}^*\otimes\wedge^{r_0-1}(\mathcal{M}^*)\rightarrow\wedge^{r_0}(\mathcal{M}^*)$$ and the natural morphism $$\wedge^{r_0}(\mathcal{M}^*)\rightarrow\Big(\wedge^{r_0}(\mathcal{M}^*)\Big)^{**}\cong\det(\mathcal{M}^*)\cong \det\mathcal{E}$$ induces a morphism
$$f:\wedge^{r_0-1}(\mathcal{M}^*)\rightarrow \mathcal{H}om(\mathcal{M}^*, \det\mathcal{E})\cong\mathcal{M}^{**}\otimes\det\mathcal{E}\cong\mathcal{M}\otimes\det\mathcal{E}.$$
It turns out that $f|_U$ is an isomorphism. Thus one infers that $$\dim \ker f=\dim\coker f=0.$$ By the same proof of (\ref{6.8}), we obtain
$$\phi_{sp}(\mathcal{M}\otimes\det\mathcal{E})\leq2.$$
Therefore, from Theorem \ref{Thm4.10} and the exact sequence
$$0\rightarrow\mathcal{M}\otimes\det\mathcal{E}\rightarrow(\det\mathcal{E})^{\oplus h^0(\mathcal{E})}\rightarrow\mathcal{E}\otimes\det\mathcal{E}\rightarrow0,$$ it follows that $$\phi_{sp}(\mathcal{E}\otimes\det\mathcal{E})\leq1.$$ This completes the proof.
\end{proof}

\begin{corollary}\label{cor6.4}
Let $Y$ be a complex smooth threefold. Assume that $\Omega^1_Y$ is generically globally generated in dimension one. If $K_Y$ is ample, then one has
$$c_3(Y)\geq3c_1(Y)c_2(Y)-\frac{26}{27}c_1^3(Y).$$
\end{corollary}
\begin{proof}
Since $\Omega_Y^1$ is $\mu_{K_Y}$-semistable, Theorem \ref{Thm6.5} gives $$\ch_3(\Omega_Y^1\otimes\omega_Y)\leq\frac{64}{54}K_Y^3.$$ This implies that
$$\ch_3(\Omega_Y^1)+\ch_2(\Omega_Y^1)K_Y+K_Y^3\leq\frac{64}{54}K_Y^3.$$
Substituting
 $$\ch_3(\Omega_Y^1)=-\frac{1}{6}(c_1^3(Y)-3c_1(Y)c_2(Y)+3c_3(Y))$$ and
$$\ch_2(\Omega_Y^1)=\frac{1}{2}c_1^2(Y)-c_2(Y)$$ into the inequality above, one obtains the conclusion.
\end{proof}

\bibliographystyle{amsplain}

\begin{thebibliography}{10}
\bibitem{Ara1}D. Arapura, Frobenius amplitude and strong vanishing theorems for vector bundles. Duke Math. J. 121 (2004), no. 2, 231--267. With
appendices by D. S. Keeler.

\bibitem{Ara2}D. Arapura, Partial Regularity and Amplitude. Amer. J. of Math. 128 (2006), no. 4, 1025--1056.

\bibitem{BMS} A. Bayer, E. Macr\`i and P. Stellari, Stability conditions on
abelian threefolds and some Calabi-Yau threefolds. Invent. Math. 206
(2016), no. 3, 869--933.

\bibitem{BMT}A. Bayer, E. Macr\`i and Y. Toda, Bridgeland stability
conditions on threefolds I: Bogomolov-Gieseker type inequalities. J.
Algebraic Geom. 23 (2014), 117--163.

\bibitem{BMSZ}M. Bernardara, E. Macr\`i, B. Schmidt and X. Zhao,
Bridgeland Stability Conditions on Fano Threefolds. \'Epijournal
Geom. Alg\'ebrique 1 (2017), Art. 2, 24 pp.

\bibitem{BS}M. C. Beltrametti and A. J. Sommese, Zero cycles and the $k$-th order embeddings of smooth
projective surfaces, (with an appendix by L. G\"ottsche), in:
Problems in the Theory of Surfaces and Their Classification, (eds:
F. Catanese, C. Ciliberto and M. Cornalba), Sympos. Math. 32 (1991),
Academic Press, London, 33--48.


\bibitem{Bog}F. A. Bogomolov, Holomorphic tensors and vector bundles on projective varieties. Izv.
Akad. Nauk SSSR. 42 (1978), 1227--1287.

%\bibitem{Bri1}T. Bridgeland, Stability conditions on triangulated categories. Ann. of Math.
%166 (2007), no. 2, 317--345.

%\bibitem{Bri2}T. Bridgeland, Stability conditions on K3 surfaces. Duke Math. J.
%141 (2008), no. 2, 241--291.

\bibitem{CK}
M.C. Chang and H. Kim, The Euler number of certain primitive Calabi–Yau threefolds.
Math. Proc. Cambridge Philos. Soc. 128 (2000), no. 1, 79--86.

\bibitem{DPS}J.P. Demailly, T. Peternell and M. Schneider, Compact complex manifolds with
numerically effective tangent bundles. J. Algebraic Geom. 3 (1994), no. 2, 295--345.

\bibitem{DRY}M. Douglas, R. Reinbacher and S.-T. Yau, Branes, Bundles and Attractors:
Bogomolov and Beyond. arXiv:math/0604597.

\bibitem{DS}R. Du and H. Sun, Inequalities of Chern classes on nonsingular projective
$n$-folds with ample canonical or anti-canonical line bundles. J. Differ. Geom. 122 (2022), 377--398.

\bibitem{En}I. Enoki, Stability and negativity for tangent sheaves of minimal Kähler spaces. (Lect. Notes Math., vol. 1339, pp. 118--126) Berlin Heidelberg New York: Springer 1988.

\bibitem{FL}M. Fulger and A. Langer, Positivity vs. slope semistability for bundles with vanishing discriminant. J. Algebra 609 (2022), 657--687.

%\bibitem{Gie}D. Gieseker, Stable vector bundles and the Frobenius morphism. Ann. Sci. Ecole Norm.
%Sup. 6 (1973), 95--101.

\bibitem{Gieseker}D. Gieseker, On a theorem of Bogomolov on Chern classes of
stable bundles. Am. J. Math. 101 (1979), 77--85.


%\bibitem{Hart}R. Hartshorne, Algebraic geometry, Graduate Texts in Mathematics
%52, Springer-Verlag, New York-Heidelberg, 1977.

\bibitem{Hart2}R. Hartshorne, Stable reflexive sheaves. Math. Ann. 254
(1980), no. 2, 121--176.

\bibitem{Hun}B. Hunt, A bound on the Euler number for certain Calabi–Yau 3-folds. J. Reine
Angew. Math. 411 (1990), 137--170.

\bibitem{HL}
D. Huybrechts and M. Lehn, The geometry of moduli spaces of sheaves.
Cambridge Mathematical Library. Cambridge University Press,
Cambridge, second edition, 2010.

\bibitem{IJL}
M. Iwai, C. Jiang, and H. Liu, Miyaoka type inequality for terminal
threefolds with nef anti-canonical divisors. Sci. China Math., to appear, 2023.

\bibitem{KW}A. Kanazawa and P.M.H. Wilson, Trilinear forms and Chern classes of Calabi-Yau threefolds. Osaka J. Math. 51 (2014), 203--215.

\bibitem{Kaw}Y. Kawamata, Boundedness of $\mathbb{Q}$–Fano threefolds. In Proceedings of the International Conference on Algebra, Part 3 (Novosibirsk, 1989), volume 131 of Contemp. Math., pages 439--445. Amer. Math. Soc., Providence, RI, 1992.

\bibitem{Ke}D. S. Keeler, Ample filters and Frobenius amplitude. J. Algebra 323 (2010), 3039--3053.

\bibitem{Kob}S. Kobayashi, Differential Geometry of Complex Vector Bundles.
Iwanami Shoten, Princeton University Press, 1987.

\bibitem{Langer1}A. Langer, Semistable sheaves in positive characteristic. Ann. Math. 159 (2004), 241--276.

\bibitem{Langer2}A. Langer, Moduli spaces of sheaves in mixed characteristic. Duke Math. J. 124 (2004), no. 3, 571--586.

\bibitem{Langer3}A. Langer, On the S-fundamental group scheme. Ann. Inst. Fourier 61 (2011), 2077--2119.

%\bibitem{Langer4}A. Langer, The Bogomolov-Miyaoka-Yau inequality for logarithmic surfaces in positive characteristic. Duke Math. J. 165
%(2016), no. 14, 2737--2769.

\bibitem{Laz}R. Lazarsfeld, Positivity in algebraic geometry. I, II, Springer-Verlag, Berlin, 2004.

\bibitem{Lee}S. Lee, Linear Systems on Fano Threefolds. I, Communications in Algebra 32 (2004), no. 7, 2711--2721.

\bibitem{Li1}C. Li, On stability conditions for the quintic
threefold. Invent. Math. 218 (2019), 301--340.

\bibitem{LL}
H. Liu and J. Liu, Kawamata–Miyaoka type inequality for $\mathbb{Q}$-Fano varieties with canonical
singularities. J. Reine Angew. Math. to appear (2023).

\bibitem{Liu}Y. Liu, Some quadratic inequalities on product varieties. arXiv:2010.14039.


\bibitem{MR}V. Mehta and A. Ramanathan, Homogeneous bundles in characteristic
$p$, in Algebraic Geometry-Open Problems (Ravello, 1982), Lecture
Notes in Math. 997 (1983), 315--320.

\bibitem{Mo1}A. Moriwaki, Relative Bogomolov's inequality and the cone of positive divisors on the moduli space of stable curves.
J. Amer. Math. Soc. 11 (1998), no. 3, 569--600.

\bibitem{Mo2}A. Moriwaki, Bogomolov conjecture over function fields for stable curves with only irreducible
fibers. Compos. Math. 105 (1997), 125--140.

\bibitem{Pao}R. Paoletti, Free pencils on divisors. Math. Ann. 303 (1995), 109--123.

\bibitem{PP}G. Pareschi and M. Popa, GV-sheaves, Fourier-Mukai transform, and Generic Vanishing. Amer. J. Math. 133 (2011), no. 1, 235--271.

\bibitem{Peter}T. Peternell, Varieties with generically nef tangent bundles. J. Eur. Math. Soc. 14 (2012), 571--603.

\bibitem{Reider}I. Reider, Vector bundles of rank 2 and linear systems on algebraic surfaces. Ann. of Math.
127 (1988), 309--316.

\bibitem{Sun}H. Sun, Tilt-stability, vanishing theorems and
Bogomolov-Gieseker type inequalities. Adv. Math. 347 (2019),
677--707.

\bibitem{Sun1}H. Sun, Stability conditions on fibred threefolds. arXiv:2201.13251.

\bibitem{Sun2}H. Sun, Bogomolov-Gieseker type inequalities on ruled threefolds. arXiv:2206.10392.

\bibitem{Sun3}H. Sun, Bogomolov's inequality for product type varieties in positive characteristic. J. Math. Soc. Japan 75 (2023), no. 1, 173--194.

\bibitem{Tan}S.-L. Tan, Cayley-Bacharach property of an algebraic variety and Fujita's conjecture. J. Alg. Geom. 9 (2000), 201--222.
\end{thebibliography}

\end{document}